\documentclass{amsart}

\usepackage{graphicx}
\usepackage{amssymb}
\usepackage{amsthm}
\usepackage{amsmath}
\usepackage{pstricks,pst-node}
\usepackage{longtable}
\usepackage{array}
\usepackage{ragged2e}
\input xy
\xyoption{all}

\swapnumbers
\newtheorem{theorem}[subsection]{Theorem}
\newtheorem{lemma}[subsection]{Lemma}

\newtheorem{corollary}[subsection]{Corollary}

\theoremstyle{definition}
\newtheorem{definition}[subsection]{Definition}
\newtheorem{remark}[subsection]{Remark}

\newcommand{\Z}{\mathbb{Z}}
\newcommand{\Q}{\mathbb{Q}}

\newcommand{\sign}{\mathrm{sign}}

\newcommand{\inv}{^{-1}}

               \def \cT {{\mathcal T}}

\def\Q{\mathbb{Q}}
\def\Z{\mathbb{Z}}

\def\inv{^{-1}}

\def\diag{\mbox{diag}\,}

\def\St{\mbox{St}}

\newcommand{\GL}{\mathrm{GL}}
\newcommand{\SL}{\mathrm{SL}}

\begin{document}

\title[On the cohomology of $\SL_n(\Z)$]
{On the cohomology of $\SL_n(\Z)$}

\author{Avner Ash} \address{Boston College\\ Chestnut Hill, MA 02445}
\email{Avner.Ash@bc.edu}



\begin{abstract}
Denote the  virtual cohomological dimension of $\SL_n(\Z)$ by
$\nu_n=n(n-1)/2$. Let  $St$ denote the Steinberg module of $\SL_n(\Q)$ tensored with $\Q$.  Let $Sh_\bullet\to St$ denote the sharbly resolution of the Steinberg module.
By Borel-Serre duality,  $H^{\nu_n-i}(\SL_n(\Z), \Q)$ is isomorphic to 
$H_{i}(\SL_n(\Z),St)$.   The latter is isomorphic to the sharbly homology 
$H_{i}((Sh_\bullet)_{\SL_n(\Z)})$.
We produce nonzero classes in 
 $H_i(\SL_n(\Z),St)$, for certain small $i$, in terms of sharbly cycles and cosharbly cocycles. 
  \end{abstract}

\maketitle
\section{Introduction}\label{intro}

Denote the  virtual cohomological dimension of $\SL_n(\Z)$ by
$\nu_n=n(n-1)/2$. 
The motivation of this paper is the conjecture of 
\cite{CFP} that 
\[
H^{\nu_n-i}( \SL_n(\Z),\Q)=0 \text{\ if\ } i<n-1. 
\]
 (See that paper for references to proofs of the
conjecture when $i=0,1$ for all $n$, and all $i$ for $n\le 7$.
The conjecture has also been proven for $i=2$, all $n$, in \cite{BMPSW}.)

Our main result is the following theorem:
\begin{theorem}\label{main}
Let $k\ge0$.  Then

(1) If $n=3k+3$, $H^{\nu_n-n}(\SL_{n}(\Z),\Q)\ne0$;

(2) If $n=3k+4$, $H^{\nu_n-(n-1)}(\SL_{n}(\Z),\Q)\ne0$.
\end{theorem}

Both (1) and (2) are implied by Theorem 2 of the research announcement~\cite{LEE}, but to the best of our knowledge, no proof of Lee's announced theorems has ever appeared.  When $k$ is even 
(1) and (2) are implied by the beautiful preprint~\cite{FB}.
 
The methods of~\cite{FB} are completely different from ours, and indeed Brown obtains many more nonzero classes.  However for odd $k$, the classes in (1) and (2) seem to be new.

Our results support the CFP conjecture, as do the announcement of Ronnie Lee and Brown's results.
\footnote{
Not too much is known about $H^*(\SL_n(\Z),\Q)$ in general. The restriction of the stable cohomology 
\[\varprojlim\limits_{n\to\infty} H^i( \SL_n(\Z), \Q)\to H^i( \SL_n(\Z), \Q)\] for a fixed $n$ was determined by Franke, as explained in  \cite{GKT}.
More classes, related to the Borel classes, appear in \cite{FB}.
Explicit computations of $H^i( \SL_n(\Z), \Q)$ for $n\le 12$ may be found in 
\cite{Soule},  \cite{LS2}, \cite{EVGS} and
\cite{SEVKM}.
Cuspidal cohomology for $\SL_n(\Z)$ for certain $n$ has recently been constructed in~\cite{BCG}.}

In this paper we  use the Borel-Serre isomorphism
\[
H^{\nu_n-i}( \SL_n(\Z), \Q)\approx H_{i}( \SL_n(\Z), St),
\] where 
$St$ denotes the Steinberg module of $\SL_n(\Q)$ tensored with $\Q$. (The notation suppresses the dependence of $St$ on $n$, which will always be clear from the context.) To study $H_{i}( \SL_n(\Z), St)$ we  use the sharbly resolution of $St$,
namely
$Sh_*\to St \to 0$ (see Section~\ref{St}).

There are three steps in the proof.  (A) We apply the ideas in
 \cite{unstable}, 
to the full group $SL_n(\Z)$.  (B) Starting from results in  \cite{AGMpuzzle}, we identify explict sharbly cycles and cocycles witnessing that Theorem~\ref{main} is true for $k=0$.
(C) We apply (A) repeatedly to (B) to obtain the theorem.

A little bit more detail:  (A) Given a cycle $z\in (Sh_i)_{\SL_m(\Z)}$ that represents a nonzero homology class in $H_{i}( \SL_m(\Z), St)$, there will be a cosharbly cocycle $\mu:Sh_i\to\Q$ such that $\mu(z)\ne0$.  Given two such cycles $z_1,z_2$ in degrees $i_1,i_2$, with corresponding cocycles 
$\mu_1,\mu_2$, for $\SL_{m_1}(\Z), \SL_{m_2}(\Z)$ respectively.  We define a composition $z=[z_1|z_2]$ which is a sharbly cycle representing a class in  $H_{i_1+i_2}( \SL_{m_1+m_2}(\Z), St)$.  When $\mu_1$ and $\mu_2$ have the same parity under the natural action by $\GL_*(\Z)$, we define a product $\mu=\mu_1\times\mu_2$, a cosharbly cocycle which pairs with $z$.  Unlike the situation in~\cite{unstable} which dealt with certain congruence subgroups, there is no reason why $\mu(z)$ should be nonzero in general.

(B)   We start with the known nonzero classes in 
$H_{\nu_n}( \SL_n(\Z), St)\approx H^{0}( \SL_n(\Z), \Q)$ for $n=1$ and $n=3$.  
To have a shot at showing that $\mu(z)\ne0$ we need to have explicit sharbly cycles $z_1,z_3$. 
For $n=1$ this is trivial, and for $n=3$ the required data is provided by~\cite{AGMpuzzle}.  

By composition we obtain an explicit sharbly cycle $z_4$ and a
cosharbly cocycle $\mu_4$ such that $\mu_4(z_4)\ne0$.  This agrees with the known result that $ H^{3}( \SL_4(\Z), \Q)=\Q$.  We explain what goes wrong (a parity mismatch) if we try use the nonzero class in 
$H_{2}( \SL_2(\Z), St)\approx H^{0}( \SL_2(\Z), \Q)$.  We also explain why we cannot compose $z_1$ and $z_4$ to get a nonzero class in $H_{3}( \SL_5(\Z), St)\approx H^{7}( \SL_5(\Z), \Q)=0$.

In (C) we continue composing cycles to obtain Theorem~\ref{main}.

I wish to thank Peter Patzt for helpful comments.

\section{The Steinberg module and the sharbly resolution}\label{St}
We quote most of this section from  \cite{AGMpuzzle}.

The \emph{Tits building} $T_n$
is the simplicial complex whose vertices are the proper nonzero
subspaces of $\Q^n$ and whose simplices correspond to flags of
subspaces.  By the
Solomon--Tits theorem $T_{n}$ has the homotopy type of a wedge of
$(n-2)$-dimensional spheres.  It is a left $\GL(n,\Q)$-module and
therefore so is its homology.  
\begin{definition}
Let $n\ge1$ and $k\ge0$.
We define the \emph{Steinberg module} $St$
to be the reduced homology of the Tits building with $\Q$-coefficients:
\[
St=\widetilde H_{n-2}(T_n,\Q).
\]
\end{definition}
Notes: (1) The Steinberg module is usually defined as the reduced homology 
with $\Z$-coefficients, but if we tensor that with $\Q$, we obtain what we are here calling $St$.

(2) $n=1$ is allowed:  $T_1$ is the empty set, and the reduced homology of the empty set is $\Q$ in dimension $-1$, so $St=\Q$ in this case.

\begin{definition}\label{sh}
The \emph{sharbly complex} $Sh_{*} $ is the following
complex of left $\GL(n,\Q)$-modules.   Let $B_{n,k}$ 
be the $\Q$-vector space
 generated by 
symbols $[v_1,\dots,v_{n+k}]$, where the $v_i$ are nonzero column vectors in
$\Q^n$.  Let $R_{n,k}$ be the subspace generated by
 the following relators:

(i) $[v_{\sigma
(1)},\dots,v_{\sigma(n+k)}]-(-1)^{\sign (\sigma)}[v_1,\dots,v_{n+k}]$ for all
permutations $\sigma$;

(ii) $[v_1,\dots,v_{n+k}]$ if $v_1,\dots,v_{n+k}$ do not span $\Q^n$; 

(iii) $[v_1,\dots,v_{n+k}]-[av_1,v_{2},\dots,v_{n+k}]$ for all $a\in
\Q^\times$.

Then $Sh_k=B_{n,k}/R_{n,k}$.

\noindent
The action of $g\in\GL(n,\Q)$ is 
given by $g[v_1,\dots,v_{n+k}]=[gv_1,\dots,gv_{n+k}]$.

\noindent If $G$ is a subgroup of $\GL(n,\Q)$, we denote the coinvariants of the $G$-action on $Sh_k$ by $(Sh_k)_G$. 
Let $C_{n,k}(G)$ be the subspace of $B_{n,k}$ generated by
the relators

(iv)  $[v_1,\dots,v_{n+k}]- [gv_1,\dots,gv_{n+k}]$ for all $g\in G$.

Then
$(Sh_k)_G=B_{n,k}/(R_{n,k}+C_{n,k}(G))$.

\noindent The boundary map $\partial \colon Sh_{k} \rightarrow
Sh_{k-1}$ is given by 
\[ 
\partial([v_1,\dots,v_{n+k}])=
\sum_{i=1}^{n+k} (-1)^{i+1}[v_1,\dots,\widehat{v_i},\dots v_{n+k}],
\]
where as usual $\widehat{v_i}$ means to delete $v_{i}$.
This induces a boundary map $\partial \colon (Sh_{k})_G \rightarrow
(Sh_{k-1})_G$.

We call an element of $Sh_k$ a ``$k$-sharbly'' and an element of the form $[v_1,\dots,v_{n+k}]$ a ``basic'' sharbly, even if
$v_1,\dots,v_{n+k}$ do not span $\Q^n$.  
\end{definition}

Theorem 5 in \cite{AGM5} immediately implies:

\begin{theorem}\label{Sh}  There is a map of $\GL(n,\Q)$-modules 
 $Sh_0 \to St$
such that 
the following is an exact sequence of $\GL(n,\Q)$-modules:
\[
\cdots\to Sh_k \to Sh_{k-1} \to \cdots \to Sh_0 
\to St \to 0.
\]
\end{theorem}

\begin{theorem}
Let $G$ be a subgroup of finite index in $\SL_n(\Z)$.  Then
$H_k(G, St)$ is isomorphic to the homology at the $k$-th place of the sequence
\[
\cdots \to Sh_{k+1}\otimes_{G} \Q \to
 Sh_k\otimes_{G} \Q \to Sh_{k-1}\otimes_{G} \Q
 \to  \cdots
\]
\end{theorem}
\begin{proof}
Because the stabilizers in $G$ of nonzero basic elements of $Sh_*$ are finite groups, the theorem follows easily from Theorem 7 of 
\cite{AGM5}.
\end{proof}

\begin{definition}
Set $[v_1,\dots,v_{n+k}]_G$ to be the image of $[v_1,\dots,v_{n+k}]$
in the coinvariants  $Sh_k\otimes_{G} \Q$.
\end{definition}

From now on set $G=\SL_n(\Z)$.

\begin{definition} A \emph{$k$-cosharbly} is a $G$-invariant linear functional 
$\mu: Sh_k\to E$
for some trivial $\Q G$-module $E$.  Thus $\mu$ 
is a linear functional $B_{n,k}\to E$ which 
vanishes on 
$R_{n,k}+C_{n,k}(G)$.
It is a \emph{cocycle} if it vanishes on $\partial([v_1,\dots,v_{n+t+1}])$ for
all nonzero $v_1,\dots,v_{n+t+1}\in \Q^n$.

A  \emph{$k$-sharbly cycle} is an element $z\in (Sh_k)_G$ such that 
$\partial z = 0$.
\end{definition}

The following is clear:
\begin{lemma}\label{cosh}
If $\mu$ is a  $k$-cosharbly cocycle and $z\in Sh_k\otimes_{G} \Q$ is a sharbly cycle and $\mu(z)\ne0$, then $z$ represents a nontrivial class in $H_k(G, St)$.
\end{lemma}

\section{Composition}\label{comp}

In this section we adapt the ideas in~\cite{unstable}, which were developed for certain congruence subgroups, to the case of
$\SL_n(\Z)$ itself.  

\begin{definition}
Let $n\ge1$ and $a+b=n$, $a,b>0$.

\noindent $\bullet$ $e_1,\dots,e_n$ denotes the standard basis of 
$\Q^n$.

\noindent $\bullet$ $U$ is the span of $e_1,\dots,e_a$ and $V$ be the span of $e_{a+1},\dots,e_n$.  

\noindent $\bullet$ The isomorphisms $f:\Q^a\to U$ and  $g:\Q^b\to V$ are defined by
$f(x)=\sum_{i=1}^a x_ie_i$ where $x$ has coordinates $x_1,\dots,x_a$,  and  by
$g(y)=\sum_{j=1}^{b} y_je_{a+j}$ where $y$ has coordinates $y_1,\dots,y_b$.

\noindent $\bullet$ $\pi:\Q^n\to \Q^b$ denotes the projection relative to the decomposition $\Q^n=U\oplus V$ composed with $g\inv$.
\end{definition}

\begin{definition}
 Suppose we have a basic sharbly
$[v_1,\dots,v_k]_{\SL_a(\Z)}$ and another 
$[w_1,\dots,w_m]_{\SL_b(\Z)}$.  Then we define
\[
[v_1,\dots,v_k|w_1,\dots,w_m]_{\SL_n(\Z)}=
[f(v_1),\dots,f(v_k), g(w_1),\dots,g(w_m)]_{\SL_n(\Z)}.
\]
More generally, if $x=\sum_ic_i[v^{i}_1,\dots,v^{i}_k]_{\SL_a(\Z)}$ and 
$y=\sum_jd_j[w^{j}_1,\dots,w^{j}_m]_{\SL_b(\Z)}$, define
\[
[x|y]_{\SL_n(\Z)}=\sum_{i,j}
c_id_j[v^{i}_1,\dots,v^{i}_k|w^{j}_1,\dots,w^{j}_m]_{\SL_n(\Z)}.
\]
We call this the  \emph{composition} of $x$ and $y$.
\end{definition}

\begin{theorem}\label{thm:comp}
(1) $[x|y]_{\SL_n(\Z)}$ is well-defined.

(2) If $x,y$ are cycles, then $[x|y]$ is a cycle.

(3) The homology classes of $[x|y]_{\SL_n(\Z)}$  and 
$(-1)^{ab}[y|x]_{\SL_n(\Z)}$  in $H_*({\SL_n(\Z)}, St)$ are equal.
\end{theorem}

\begin{proof}
(1) All tensor products will be over $\Q$.  
Let's make the following abbreviations: $B_1=B_{a,k}$,
$B_2=B_{b,m}$, $B_3=B_{n,k+m}$;
$W_1=R_{a,k}+C_{a,k}({\SL_a(\Z))}$,
$W_2=R_{b,m}+C_{b,m}({\SL_b(\Z))}$,
$W_3=R_{n,k+m}+C_{n,k+m}({\SL_n(\Z))}$.
Then $(Sh_k)_{\SL_a(\Z)}=B_1/W_1$,
$(Sh_m)_{\SL_b(\Z)}=B_2/W_2$,
$(Sh_{m+k})_{\SL_n(\Z)}=B_3/W_3$.

Define $\phi:B_1\otimes B_2\to B_3$ by
\[
\phi([v_1,\dots,v_k]\otimes [w_1,\dots,w_m])=
[f(v_1),\dots,f(v_k), g(w_1),\dots,g(w_m)].
\]
This induces a map
$\hat \phi:B_1\otimes B_2\to B_3/W_3$.
The kernel of $\hat\phi$ contains $W_1\otimes B_2+
B_1\otimes W_2$.  
Let's check this for $W_1\otimes B_2$; the proof for 
$B_1\otimes W_2$ is similar.

 (i) $[v_{\sigma
(1)},\dots,v_{\sigma(k)}, w_1,\dots,w_m]-
(-1)^{\sign (\sigma)}[v_1,\dots,v_{k}, w_1,\dots,w_m]\in W_3
$
because $\sigma$ may also be viewed as inducing a permutation of 
all the $k+m$ vectors;

(ii) If $v_1,\dots,v_{k}$ do not span $\Q^a$ then
$[v_1,\dots,v_{k}, w_1,\dots,w_m]\in W_3$ because
$v_1,\dots,v_{k}, w_1,\dots,w_m$ do not span $\Q^n$;

(iii) $[v_1,\dots,v_{k},w_1,\dots,w_m]-[av_1,v_{2},\dots,v_{k},w_1,\dots,w_m] \in W_3$, clearly;

(iv) For any $g\in\SL_a(\Z)$, we can lift $g$ to 
an element $g'\in\SL_n(\Z)$ with the properties that $g'|U=g$ and $g'|V=id_V$.
Therefore,
$[gv_1,\dots,gv_{k},w_1,\dots,w_m]
-[v_1,\dots,v_{k},w_1,\dots,w_m]\in W_3$.

Thus
$\hat\phi$ induces a well-defined map
$ \psi:B_1/W_1\otimes B_2/W_2\to B_3/W_3$
and clearly $ \psi(x\otimes y)=[x|y]_{\SL_n(\Z)}$.

(2) Suppose $x=\sum_ic_i[v^{i}_1,\dots,v^{i}_k]_{\SL_a(\Z)}$ and 
$y=\sum_jd_j[w^{j}_1,\dots,w^{j}_m]_{\SL_b(\Z)}$
are cycles. Then
\begin{gather*}
\partial [x|y]_{\SL_n(\Z)}=\sum_{i,j}
c_id_j
\Big( 
\sum_{\alpha=1}^{k} (-1)^{\alpha+1}
[v^{i}_1,,\dots,\widehat{v^{i}_\alpha},\dots,v^{i}_k|w^{j}_1,\dots,w^{j}_m]_{\SL_n(\Z)}+\\
(-1)^k
\sum_{\beta=1}^{m} (-1)^{\beta+1} 
[v^{i}_1,\dots,v^{i}_k|w^{j}_1,\dots,\widehat{w^{j}_\beta},\dots,  
w^{j} _m]_{\SL_n(\Z)} \Big)
.
\end{gather*}
This equals
\begin{gather*}
[\sum_{i}\sum_{\alpha=1}^{k} (-1)^{\alpha+1}c_i[v^{i}_1,\dots,\widehat{v^{i}_\alpha},\dots,v^{i}_k]|
\sum_{j}d_j
[w^{j}_1,\dots,w^{j}_m]]_{\SL_n(\Z)}+\\
(-1)^k
[\sum_{i}c_i[v^{i}_1,\dots,v^{i}_k]|
\sum_{\beta=1}^{m} (-1)^{\beta+1}
d_j[w^{j}_1,\dots,\widehat{w^{j}_\beta},\dots,w^{j}_m]]_{\SL_n(\Z)}=0
\end{gather*}
because $x$ and $y$ are cycles.  

(3) Without loss of generality, either $a$ or $b$ or both is greater than 1.  
Let $U'$ be the span of $e_1,\dots,e_b$ and $U$ be the span of $e_{b+1},\dots,e_n$.   In the construction of $[y,x]_{\SL_n(\Z)}$, $U',V'$ play the roles of $U,V$ respectively.

Choose an even permutation $\sigma$ of $\{1,\dots,n\}$ such that 
$U'$ is the span of the vectors $e_{\sigma{(a+1)}},\dots,e_{\sigma n}$, and $V'$ be the span of $e_{\sigma{1}},\dots,e_{\sigma a}$.     Let $g_\sigma\in\SL_n(\Z)$ be the matrix such that $g_\sigma e_i=e_{\sigma i}$.
 Thus $g_\sigma U=V'$ and $g_\sigma V=U'$.

Using property (i) in the definition of the sharbly complex, we see that
 \[
 (g_\sigma[x,y])_{\SL_n(\Z)}=(-1)^{ab}[y,x]_{\SL_n(\Z)}.
 \]
 
 On the other hand, $g_\sigma$ acts on $H_*({\SL_n(\Z)}, St)$ via the conjugation action of $g_\sigma$
 on the group ${\SL_n(\Z)}$ and the usual action of $g_\sigma$ on the coefficients $St$.  Since $g_\sigma\in\SL_n(\Z)$, this action is trivial.  It can be computed on the chain level by taking the projective resolution $Sh_\bullet\to\St$ and letting $g_\sigma$ act on $Sh_\bullet$ in the usual way.  Therefore
  \[
 (g_\sigma[x,y])_{\SL_n(\Z)}=[x,y]_{\SL_n(\Z)}
 \]
 in homology.
\end{proof}

\section{Product}\label{prod}
In this section we define the product of two cosharblies and show that if they are both cocycles, then so is their product.   
We continue the notation of the preceding section, so $a+b=n$, $U\oplus V = \Q^n$.  Note that $\SL_n(\Z)$ acts transitively on the $a$-dimensional subspaces of $\Q^n$.  

\begin{definition}
A set of vectors $S=\{x_1,\dots,x_r\}\subset\Q^n$ is \emph{pliable} if and only if it spans a subspace of dimension $a$.  We also say ``the vectors 
$x_1,\dots,x_r$ are pliable''.
If they are pliable 
and $\gamma\in\SL_n(\Z)$ satisfies $\gamma x_1,\dots,\gamma x_r\in U$, we say $\gamma$ ``plies'' $S$ and 
$\gamma$ ``plies'' $x_1,\dots,x_r$.
\end{definition}
Of course, if $\gamma$ plies  $x_1,\dots,x_r$ then $\gamma x_1,\dots,\gamma x_r$ span $U$.

\begin{definition}
Let $\xi$ be a $(t-c)$ cosharbly for $\SL_c(\Z)$.  Let $s_c=\diag(-1,1,\dots,1)\in\GL_c(\Z)$.
Say that $\xi$ ``has parity'' if there exists a value of $e=0,1$ such that for any sharbly 
$[x_1,\dots,x_t]_{\SL_c(\Z)}$, 
\[
\xi([s_cx_1,\dots,s_cx_t]_{\SL_c(\Z)})
=(-1)^e\xi([x_1,\dots,x_t]_{\SL_c(\Z)}).
\]
If $\xi$ has parity, say that $\xi$ is even if $e=0$ and odd if $e=1$.  We also say ``$\xi$ has parity $e$''.
\end{definition}

\begin{definition}

Let $\mu$ be a $(k-a)$-cosharbly for $\SL_a(\Z)$ and 
$\nu$ be an $(m-b)$-cosharbly for $\SL_b(\Z)$.    Assume that they each have parity and that they have the same parity.

To make the notation lighter, we will sometimes identify $U$ with $\Q^a$ via the isomorphism $f$.
Define the $(k+m-n)$-cosharbly $\mu\times\nu$ for $\SL_n(\Z)$ by
\begin{gather*}
(\mu\times\nu)([v_1,\dots,v_k,v_{k+1},\dots,v_{k+m}]_{\SL_n(\Z)})=\\
\sum_S \sign(\sigma) 
\mu([\gamma v_{\sigma1},\dots,\gamma v_{\sigma k}]_{\SL_a(\Z)})
\nu([\pi(\gamma v_{\sigma(k+1)}),\dots,
\pi(\gamma v_{\sigma {(k+m)}})]_{\SL_b(\Z)}),
\end{gather*}
and extended by linearity, with the following notation: 
 \begin{quote}
The sum runs over pliable subsets $S$ of
 $\{v_1,\dots,v_k,v_{k+1},\dots,v_{k+m}\}$ of cardinality $k$.  For each such $S$, choose a permutation $\sigma\in S_n$ such that
 $S=\{v_{\sigma1},\dots, v_{\sigma k}\}$ 
 and choose an element $\gamma\in\SL_n(\Z)$ that plies $S$. 
 \end{quote} 
 
 In this definition, we extend $\nu$ so that it equals zero if any of its arguments is 0.  This will happen if and only if for 
 some $\sigma i \not\in \{{\sigma1},\dots,{\sigma k}\}$,  
 $v_{\sigma i}$ is in the span of 
 $v_{\sigma1},\dots,v_{\sigma k}$.
\end{definition}

\begin{remark}
If $\alpha\in\SL_n(\Z)$ stabilizes $U$ then in block form 
\[\alpha=\begin{bmatrix}
\alpha_1&*\\
0&\alpha_2
\end{bmatrix}\] with $\alpha_1\in \GL_a(\Z)$ 
and $\alpha_2\in \GL_b(\Z)$, and 
$\det \alpha_1=\det\alpha_2$.
Therefore
$\alpha$ maps $U\to U$ via multiplication by $f\circ\alpha_1\circ f\inv$ and
if $y\in\Q^n$  then 
$\pi(\alpha y) = \alpha_2\pi(y)$.
\end{remark}

\begin{theorem}\label{times1}
The definition of $\mu\times\nu$ is independent of choices and
defines a cosharbly.
\end{theorem}

\begin{proof}
We first show the definition is independent of choices.  
Suppose $\mu$ and $\nu$ both have parity $e$.

Suppose we choose $\delta$ instead of $\gamma$.  Use angle brackets to denote $\Q$-span.  Then
$\langle\gamma S\rangle = \langle\delta S\rangle =  U.$
Therefore $\alpha:=\gamma\delta\inv$ stabilizes $U$.  
Let $\lambda(e,f)=1$ unless $e=1$ and $f=-1$, in which case 
$\lambda(1,-1)=-1$.
It follows that
\begin{align*}
\mu(
[\delta v_{\sigma1},\dots,\delta v_{\sigma k}]_{\SL_a(\Z)})=\\
\lambda(e,\det \alpha_1)\mu([\alpha \delta v_{\sigma1},\dots,\alpha\delta  v_{\sigma k}]_{\SL_a(\Z)})=\\\lambda(e,\det \alpha_1)\mu(
[\gamma v_{\sigma1},\dots,\gamma v_{\sigma k}]_{\SL_a(\Z)})
\end{align*}
and
\begin{align*}
\nu([\pi(\delta v_{\sigma(k+1)}),\dots,
\pi(\delta v_{\sigma {(k+m)}})]_{\SL_b(\Z)})=\\
\lambda(e,\det \alpha_2)\nu([\alpha_2\pi(\delta v_{\sigma(k+1)}),\dots,
\pi(\alpha_2\delta v_{\sigma {(k+m)}})]_{\SL_b(\Z)})=\\
\lambda(e,\det \alpha_2)\nu([\pi(\alpha\delta v_{\sigma(k+1)}),\dots,
\pi(\alpha\delta v_{\sigma {(k+m)}})]_{\SL_b(\Z)})=\\
\lambda(e,\det \alpha_2)\nu([\pi(\gamma v_{\sigma(k+1)}),\dots,
\pi(\gamma v_{\sigma {(k+m)}})]_{\SL_b(\Z)}).
\end{align*}
So the value of the formula for $\mu\times\nu$ is the same for $\delta$ 
as for $\gamma$, 
since $\det\alpha_1=\det\alpha_2$.

Suppose $\tau$ satisfies $S=\{v_{\tau1},\dots, v_{\tau k}\}$.  Then $\psi:=\tau\sigma\inv$ stabilizes 
$\{\sigma1,\dots,\sigma k\}$ and
$\{\sigma(k+1),\dots,\sigma {(k+m)}\}$.
We can choose the same $\gamma$ for $\sigma$ and $\tau$ and compare the terms corresponding to $S$:
\begin{gather*}
\sign(\sigma)
\mu([\gamma v_{\sigma1},\dots,\gamma v_{\sigma k}]_{\SL_a(\Z)})\nu([\pi(\gamma v_{\sigma(k+1)}),\dots,\pi(\gamma v_{\sigma {(k+m)}})]_{\SL_b(\Z)})=\\
\sign(\sigma)\sign(\psi|\{1,\dots,k\})\sign(\psi|\{k+1,\dots,{k+m}\})\times\\
\mu([\gamma v_{\psi\sigma1},\dots,\gamma v_{\psi\sigma k}]_{\SL_a(\Z)})\nu([\pi(\gamma v_{\psi\sigma(k+1)}),\dots,\pi(\gamma v_{\psi\sigma {(k+m)}})]_{\SL_b(\Z)})=\\
\sign(\tau) 
\mu([\gamma v_{\tau1},\dots,\gamma v_{\tau k}]_{\SL_a(\Z)})\nu([\pi(\gamma v_{\tau(k+1)}),\dots,\pi(\gamma v_{\tau {(k+m)}})]_{\SL_b(\Z)}).
\end{gather*}

Next we have to show that $\mu$ vanishes on all the relators (i)-(iv):

(i) Let $\phi\in S_n$.  
Then
\begin{gather*}
(\mu\times\nu)(v_{\phi1},\dots,v_{\phi {(k+m)}})=\\
\sum_S \sign(\sigma) 
\mu([\gamma v_{\sigma\phi1},\dots,\gamma v_{\sigma \phi k}]_{\SL_a(\Z)})
\nu([\pi(\gamma v_{\sigma\phi(k+1)}),\dots,
\pi(\gamma v_{\sigma \phi {(k+m)}})]_{\SL_b(\Z)})
\end{gather*}
where $\gamma$ plies $S=\{v_{\sigma\phi1},\dots,v_{\sigma\phi k}\}$.

Let $\sigma'=\phi\inv\sigma\phi$.  
Then $\gamma$ plies $S=
\{v_{\phi\sigma'1},\dots,v_{\phi\sigma' k}\}$.
So then the right hand side equals 
\[
 \sum_S \sign(\phi)\sign(\phi\sigma') 
\mu([\gamma v_{\phi\sigma'1},\dots,\gamma v_{\phi\sigma' k}]_{\SL_a(\Z)})
\nu([\pi(\gamma v_{\phi\sigma'(k+1)}),\dots,
\pi(\gamma v_{\phi\sigma' {(k+m)}})]_{\SL_b(\Z)}.
\]
Since we can choose $\phi\sigma'$ as the permutation for $S$, this in turn equals
\[
\sign(\phi)
(\mu\times\nu)([v_1,\dots,v_k,v_{k+1},\dots,v_{(k+m)}]_{\SL_n(\Z)}).
\]

(ii) If $v_1,\dots,v_{{k+m}}$ do not span $\Q^n$ then
either there is no pliable $S$, or if $S$ is pliable, the corresponding term in the formula for 
$(\mu\times\nu)$ equals 0  because 
$\pi(\gamma v_{\sigma(k+1)}),\dots,
\pi(\gamma v_{\sigma {(k+m)}})$ cannot span a $b$-dimensional space.

(iii) It is clear that for any $a_i\in\Q$,
$(\mu\times\nu)([a_1v_1,\dots,a_kv_{k},a_{k+1}v_{k+1},\dots,a_{k+m}v_{k+m}]_{\SL_n(\Z)}=
(\mu\times\nu)[v_1,v_{2},\dots,v_{k},v_1,\dots,v_{k+m}]_{\SL_n(\Z)}$.

(iv) For any $g\in\SL_n(\Z)$,
\begin{gather*}
(\mu\times\nu)([gv_1,\dots,gv_k,gv_{k+1},\dots,gv_{k+m}]_{\SL_n(\Z)})=\\
\sum_S \sign(\sigma) 
\mu([\gamma gv_{\sigma1},\dots,\gamma gv_{\sigma k}]_{\SL_a(\Z)})
\nu([\pi(\gamma gv_{\sigma(k+1)}),\dots,
\pi(\gamma gv_{\sigma (k+m)})]_{\SL_b(\Z)})
\end{gather*}
where $\gamma$ plies $\{gv_{\sigma1},\dots,gv_{\sigma k}\}$.
On the other hand, since $\gamma g$ plies $\{v_{\sigma1},\dots,v_{\sigma k}\}$, we have 
\begin{gather*}
(\mu\times\nu)([v_1,\dots,v_k,v_{k+1},\dots,v_{k+m}]_{\SL_n(\Z)})=\\
\sum_S \sign(\sigma) 
\mu([\gamma gv_{\sigma1},\dots,\gamma gv_{\sigma k}]_{\SL_a(\Z)})
\nu([\pi(\gamma gv_{\sigma(k+1)}),\dots,
\pi(\gamma gv_{\sigma ({k+m)}})]_{\SL_b(\Z)}),
\end{gather*}
which equals what we had before.
\end{proof}

\begin{theorem}\label{times2}
If $\mu$ and $\nu$ are cocycles, so is $\mu\times\nu$.
\end{theorem}

\begin{proof}
If $v_1,\dots,v_r$ are a sequence of vectors in $\Q^n$, define a ``palette'' to be a maximal pliable subset of them.  Any pliable subset is contained in a unique palette.

Let $r=k+m+1$. We are to show that 
\[
(\mu\times\nu)(\partial [v_1,\dots,v_r]_{\SL_n(\Z)} )= 0.
\]
The left hand side equals
\[
\sum_{j=1}^r (-1)^{j+1}\sum_S \sign(\sigma)
\mu([\gamma v^j_{\sigma1},\dots,\gamma v^j_{\sigma k}]_{\SL_a(\Z)})
\nu([\pi(\gamma v^j_{\sigma(k+1)}),\dots,
\pi(\gamma v^j_{\sigma {(k+m)}})]_{\SL_b(\Z)})
\]
where $(v^j_1,\dots,v^j_{r-1})=(v_1,\dots,\widehat v_j, \dots,v_r)$.

We will group the terms by palette, and show that the sum of each group vanishes.  For the $\sigma$ used in the term corresponding to $S$, we will choose the $(k,m)$-shuffle permutation such that
$\{v_{\sigma1},\dots,v_{\sigma k} \} = S$.
(A shuffle permutation is one satifying $\sigma(i)<\sigma(\ell)$ if $1\le i<\ell\le k$
and $\sigma(i)<\sigma(\ell)$ if $k+1\le i<\ell\le k+m$.)

Fix a palette $T$ and let $\cT_T$ be the set of all 
$\{v^j_{\sigma1},\dots, v^j_{\sigma k}\}$ which are subsets of $T$.   Define $\Xi_T=$
\[
\sum_{\{v^j_{\sigma1},\dots, v^j_{\sigma k}\}\in\cT_T}
(-1)^{j+1} \sign(\sigma)
\mu([\gamma v^j_{\sigma1},\dots,\gamma v^j_{\sigma k}]_{\SL_a(\Z)})
\nu([\pi(\gamma v^j_{\sigma(k+1)}),\dots,
\pi(\gamma v^j_{\sigma {(k+m)}})]_{\SL_b(\Z)}).
\]
Then $(\mu\times\nu)(\partial [v_1,\dots,v_r]_{\SL_n(\Z)} )=\sum_T \Xi_T$.
By definition, $|T|\ge k$.

Case 1: $|T|>k+1$.  In each term of the sum,
 there is some $\ell>k$ such that
$v^j_{\sigma\ell}$ is in the span of 
$v^j_{\sigma1},\dots, v^j_{\sigma k}$.
Therefore $\pi(v^j_{\sigma\ell})=0$, so that $\nu=0$ and each term of the sum equals 0.

Case 2: $|T|=k$.  In this case there is only one $S$, namely $S=T$.  There will be $r-k=m+1$ terms in the sum, namely one term for each
$j$ such that $v_j\not\in T$.  For each such $j$ we have
$v^j_{\sigma1},\dots, v^j_{\sigma k}$ constant, independent of $j$, since they must constitute the set $T$ and their subscripts must be in increasing order, since $\sigma$ is a shuffle.  Call these constant values $v_{\sigma1},\dots, v_{\sigma k}$.

Let $\tau$ be the $(k,r-k)$-shuffle permutation of $\{1,\dots,r\}$ that agrees with $\sigma$ on $1,\dots,k$.
List the vectors in $\{v_1,\dots,v_r\}-T$
in increasing order of subscripts:
$v_{\tau (k+1)},\dots, v_{\tau (r)}$.
Then in each term, the omitted vector $v_j$ will be one of these, say 
$v_{\tau (k+x)}$.  In other words, $j=\tau(k+x)$. 

Then
\[
(-1)^j
(v^j_{\sigma(k+1)}),\dots,v^j_{\sigma {(k+m)}})=
(-1)^{\tau (k+x)+1}
( v_{\tau (k+1)},\dots, \widehat{v_{\tau (k+x)}},\dots, v_{\tau (r)}).
\]
Then, using the same $\gamma$ for all terms,
$\Xi_T=$
\begin{gather*}
\mu([\gamma v_{\sigma1},\dots,\gamma v_{\sigma k}]_{\SL_a(\Z)})\times\\
\sum_{x=1}^{m+1}
(-1)^{\tau (k+x)+1}\sign(\sigma)
\nu([\pi(\gamma v_{\tau (k+1)}) ,\dots,
\widehat{\pi(\gamma v_{\tau (k+x)})},\dots
\pi(\gamma v_{\tau (k+m+1)})
]_{\SL_b(\Z)}).
\end{gather*}
Note that $\sigma$ is now a function of $j$ and so a function of $x$.
None of the arguments of $\nu$ vanish, because $T$ is a palette.
We proceed to determine the sign of each term.

Call the vectors whose indices are $\sigma(1),\dots,\sigma(k)$ ``red'' and call the other vectors ``blue''.  
The shuffle permutation $\tau$ pushes all the red vectors to the left past all the blue vectors, without changing the order of the red vectors nor of the blue vectors.  

Looking at the term in $\Xi_T$ coming from $j=\tau(k+x)$, note that $v_j$ is a blue vector. Let $f(x)$ be the number of red vectors to the right of $v_j$.  We delete $v_j$ and perform the corresponding shuffle permutation $\sigma$.  The only difference between $\sigma$ and $\tau$ is that $v_j$ is missing.  The $f(x)$ red vectors to the right of $v_j$ don't need to jump over $v_j$.  So $\sign(\sigma)=(-1)^{f(x)}\sign(\tau)$.
 Thus
 $\Xi_T=$
\begin{gather*}
\sign(\tau)\mu([\gamma v_{\sigma1},\dots,\gamma v_{\sigma k}]_{\SL_a(\Z)})\times\\
\sum_{j\in A} (-1)^{\tau (k+x)+1} (-1)^{f(x)}
\nu([\pi(\gamma v_{\tau (k+1)}) ,\dots,
\widehat{\pi(\gamma v_{\tau (k+x)})},\dots
\pi(\gamma v_{\tau (k+m+1)})
]_{\SL_b(\Z)}).
\end{gather*}
Because $\nu$ is a cocycle, and none of its arguments is 0, it will follow
 that $\Xi_T=0$ if we show that the sign 
\[
(-1)^{\tau(k+x)+f(x)}
\]
alternates as $x$ increases from 1 to 2, from 2 to 3, and so one until it goes from $m$ to $m+1$.  Let $1\le x\le m$ and $j=\tau(k+x)$.
We are starting with the blue vector $v_j$.  Suppose the next blue vector after it is $v_{j+\alpha}$.

(i) If $\alpha=1$,  $v_{j+1}$ is the very next vector.  
Then  $\tau(k+x+1)=\tau(k+x)+1$ has increased by 1 and $f(x+1)=f(x)$ remains constant.

(ii) If  $\alpha>1$, there are $\alpha-1$ red vectors between 
$v_j$ and $v_{j+\alpha}$. Then 
$\tau(k+x+1)=\tau(k+x)+\alpha$ has increased by $\alpha$, 
but $f(x+1)=f(x)-(\alpha-1)$ has decreased by $\alpha-1$.
This finishes Case 2.

Case 3: $|T|=k+1$.  Now there are $k+1$ $S$'s
in $\cT_T$.  Let $T=\{v_{c_1},\dots,v_{c_{k+1}}\}$, where the subscripts are in increasing order.  
Let $\tau$ be the $(k+1,r-k-1)$ shuffle that takes $(1,\dots,k+1)$ to $(c_1,\dots,c_{k+1})$.  Then the elements of $\cT_T$ are
$S_\ell=T-\{v_{\tau(\ell)}\}$, $\ell=1,\dots,k+1$.

Recall that $\Xi_T=$
\begin{gather*}
\sum_{\{v^j_{\sigma1},\dots, v^j_{\sigma k}\}\in\cT_T}
(-1)^{j+1}\sign(\sigma)
\mu([\gamma v^j_{\sigma1},\dots,\gamma v^j_{\sigma k}]_{\SL_a(\Z)})
\times\\
\nu([\pi(\gamma v^j_{\sigma(k+1)}),\dots,
\pi(\gamma v^j_{\sigma {(k+m)}})]_{\SL_b(\Z)}),
\end{gather*}
and our task to show that this equals 0.  

Unlike Case 2, in the current case, because $T$ has $k+1$ elements,  the deleted vector $v_j$ in a given term of $\Xi_T$ could a priori be anything.  However, if we delete $v_j\not\in T$, then 
$\nu$ will have to involve an argument of the form
$\pi(\gamma v_d)$ where $v_d\in T$.  Then $\gamma v_d$ will be 
 linearly dependent on the arguments of $\mu$ and 
$\pi(\gamma v_d)=0$ .
Hence, the argument of $\nu$ will contain a 0 vector, and the corresponding term will vanish.  So we only need to retain terms where $v_j$ is deleted from $T$.

Now
$S_\ell={\{v^j_{\sigma1},\dots, v^j_{\sigma k}\}\in\cT_T}$ arises by deleting $v_j$ from $v_1,\dots,v_r$, where 
$j=\tau(\ell)$.  In this case, $\sigma$ is the $(k,r-1-k)$-shuffle $\sigma_\ell$ such that 
$\{\sigma(1), \dots, \sigma(k)\}=
\{\tau(1), \dots, \widehat{\tau(\ell)}, \dots, \tau(k+1)\}$.
The other indices $\sigma_\ell(k+1)=\tau_\ell(k+2),\dots,\sigma_\ell(k+m)
=\tau_\ell(k+m+1)$ are independent of $\ell$.

Using the same $\gamma$ to ply each $S_\ell$, we have that
 $\Xi_T=$
\begin{gather*}
\nu([\pi(\gamma v^j_{\tau(k+2)}),\dots,
\pi(\gamma v^j_{\tau {(k+m+1)}})]_{\SL_b(\Z)})\times\\
\sum_{\ell=1}^{k+1}
(-1)^{\tau(\ell)+1}\sign(\sigma_\ell)
\mu([\gamma v_{\tau1},\dots,
\widehat{\gamma v_{\tau \ell}},\dots,
\gamma v_{\tau (k+1)}]_{\SL_a(\Z)}).
\end{gather*}

Because $\mu$ is a cocycle, and none of its arguments is 0, it will follow that $\Xi_T=0$ if we show that the sign 
\[
(-1)^{\tau(\ell)} \sign(\sigma_\ell)
\]
alternates as we go from $\ell$ to $\ell+1$, for $\ell=1,\dots,k$.

Call the vectors whose indices are $\sigma(1),\dots,\sigma(k+1)$ ``red'' and call the other vectors ``blue''.  The \emph{valence} of a red vector is the number of blue vectors to its left, when we list the vectors in the order 
$v_1,v_2,\dots,v_{k+1+m}$.

The parity of $\tau$ is the parity of the sum of all the valences of all the red vectors, because the shuffle permutation $\tau$ starts with the first red vector and jumps it over all the blue vectors to its left, then the second red vector jumps over all the blue vectors to its left and so on.

The shuffle $\sigma_\ell$ is the same as $\tau$ except that the red vector $v_{\tau(\ell)}$ has been deleted.  Therefore the
sign of $\sigma_\ell$ is the sign of $\tau$  times 
$(-1)^{\text{valence of  the deleted red vector}} $.
 Since the sign of $\tau$ is constant, we are reduced to proving that 
\[
(-1)^{\tau(\ell)}(-1)^{\text{valence of  }v_{\tau(\ell)} }
\]
 alternates sign.
 
Extend the definition of $\tau$ so that $\tau(0)=0$ and let $v_0$ denote a ``ghost red vector'' to the left of all the other vectors.
The number of blue vectors between the red vectors 
$v_{\tau(\ell)}$ and $v_{\tau(\ell-1)}$
is exactly $\tau(\ell)-\tau(\ell-1)-1$.
 Therefore the valence of $v_{\tau(\ell)}$ is 
 \[
 (\tau(\ell)-\tau(\ell-1)-1)+(\tau(\ell-1)-\tau(\ell-2)-1)+\dots+
 (\tau(1)-\tau(0)-1)
 \]
 which equals $\tau(\ell)-\ell$.
 So
 \[
(-1)^{\tau(\ell)}(-1)^{\text{valence of  }v_{\tau(\ell)} }=
(-1)^{\tau(\ell)}(-1)^{\tau(\ell)}(-1)^\ell,
\]
 which does alternate sign as $\ell$ increases.
 \end{proof}
 
 \begin{corollary}
Suppose $z_1$ represents a class in $H_i(\SL_a(\Z),St)$, and
$z_2$ represents a class in $H_j(\SL_b(\Z),St)$, and
 $\mu_1$ is a cosharbly cocycle such that $\mu_1(z_1)$ is defined, and
  $\mu_2$ is a cosharbly cocycle such that $\mu_2(z_2)$ is defined.  Assume that $\mu_1$ and $\mu_2$ have the same parity and that 
$(\mu_1\times\mu_2)([z_1|z_2]_{\SL_{a+b}(\Z)})\ne0$.
Then
$[z_1|z_2]_{\SL_{a+b}(\Z)}$ represents a nonzero class 
in $H_{i+j}(\SL_{a+b}(\Z),St)$.
 \end{corollary}
 
 \begin{proof}
 By the theorems, 
  $\mu_1\times\mu_2$ is a sharbly cocycle, and 
  $[z_1|z_2]_{\SL_{a+b}(\Z)}$ is a sharbly cycle, so if they pair to something nonzero,   $[z_1|z_2]_{\SL_{a+b}(\Z)}$ represents a nonzero class in $H_{*}(\SL_{a+b}(\Z),St)$.  To determine $*$, note that
  $z_1$ is in $(Sh_{i})_{\SL_a(\Z)}$ and
  $z_2$ is in $(Sh_{j})_{\SL_b(\Z)}$, so that
    $[z_1|z_2]_{\SL_{a+b}(\Z)}$ is in $(Sh_{i+j})_{\SL_{a+b}(\Z)}$.
 \end{proof}
 
 \section{Starter cases: $n=1,3$}
 To get started we need some sharbly cycles and cocycles.
 
 First let $n=1$.  Then $\SL_1(\Z)$ is the trivial group, and $H_0(\SL_1(\Z),St)=H^0(\SL_1(\Z),\Q)=\Q$.
 The sharbly $z_1:=[1]_{\SL_1(\Z)}$ is a cycle, and the cosharbly 
 $\mu_1$ that takes $[1]$ to 1 is a cocycle, and $\mu_1(z_1)=1$.  Therefore 
 $z_1$ generates $H_0(\SL_1(\Z),St)$.  This proves the rather trivial theorem:

 \begin{theorem}
(1) $ \mu_1(z_1)\ne0$.
(2) $H_0(\SL_1(\Z),St)\ne0$.
(3) $H^0(\SL_1(\Z),\Q)\ne0$.
 \end{theorem}
 \noindent Note that $\mu_1$ has even parity.

Next, let $n=3$.  Let 
\[
 z_3:=\begin{bmatrix}
 1&0&0&1&0&1\\
  0&1&0&-1&1&0\\
   0&0&1&0&-1&-1\\
 \end{bmatrix}_{\SL_3(\Z)}.
  \]
 
  \begin{theorem}
(1) There exists a cosharbly cocycle $\mu_3$ such that 
$ \mu_3(z_3)\ne0$.
(2) $H_3(\SL_3(\Z),St)\ne0$.
(3) $H^0(\SL_3(\Z),\Q)\ne0$.
 \end{theorem}
 \begin{proof}

  In~\cite{AGMpuzzle} it is proven that $z_3$ is a cycle
    (which is not hard to do by hand) and that there exists a sharbly cocycle 
   $\mu_3$ such that $\mu_3(z_3)\ne0$. 
  (Up to a nonzero factor $\mu_3([a_1,\dots,a_6]_{\SL_3(\Z)})$ is the volume of the interior of the convex hull of  
  $\{a_i{}^ta_i\}$ in the cone of positive-definite symmetric matrices.)
   \footnote{In general, one can take the perfect quadratic form called $A_n$, and its minimal vectors $v_1,\dots,v_{\nu_n}$, and form the sharbly
  $z=[v_1,\dots,v_{\nu_n}]_{\SL_n(\Z)}$.    When $n=3$, $z=z_3$.  In~\cite{AGMpuzzle} we show that if $n =2,3$ then $z$ is a sharbly cycle, but this is a low-dimensional accident.  If $n>3$, $z$ is not a sharbly cycle.   The reason for this is that the nearest Voronoi neighbors of the perfect form $A_n$ are all isomorphic to the perfect form $D_n$ if $n>2$.  Now $A_3=D_3$, but if $n>3$ then $A_n$ and $D_n$ are not isomorphic.  This implies that the automorphism group of a facet of the top-dimensional Voronoi cell corresponding to $A_n$ is a subgroup of the 
  automorphism group of $A_n$ itself, which in turn implies that
  $z$ is not a sharbly cycle, since any element in $\SL_n(Z)$ preserves orientation in the top-dimension. 
  }

  Therefore 
 $z_3$ generates $H_3(\SL_3(\Z),St)$.  Then (3) follows by Borel-Serre duality, although, of course, it is trivially true.
 \end{proof}
 \noindent In the next section we will see that $\mu_3$ has even parity.

 \section{Parity}
Recall
that $s_n=\diag(-1,1,\dots,1)\in \GL_n(\Z)$.  The 2-element group $\GL_n(\Z)/ \SL_n(\Z)$ is generated by $s_n$.  It acts on the space of cosharblies by $(s_n\mu)(z)=\mu(s_nz)$.
\begin{lemma}
Let $z$ be a sharbly cycle, $\mu$ a cosharbly cocycle, such that $\mu(z)\ne0$.
Let $\epsilon=1$ or $-1$ and suppose that $s_nz=\epsilon z$.
Then $\mu$ has even parity if $\epsilon=1$ and odd parity if $\epsilon=-1$.
\end{lemma}

\begin{proof}
Let $\mu^+=(\mu+s_n\mu)/2$ and $\mu^-=(\mu-s_n\mu)/2$, so that $\mu^+$ has even parity, $\mu^-$ has odd parity, 
and $\mu^++\mu^-=\mu$.  Then either
 $\mu^+(z)\ne0$ or  $\mu^-(z)\ne0$ or both.
Suppose   $\mu^+(z)\ne0$.  Then
$\mu^+(z)=\mu^+(s_n z)=\epsilon \mu^+(z)$.  It follows that 
$\epsilon =1$.
Then  $\mu^-(z)=-\mu^-(s_n z)=-\epsilon \mu^-(z)
=-\mu^-(z)$ so that $\mu^-=0$.  It follows that $\mu=\mu^+$ has even parity.  The argument if $\mu^-(z)\ne0$ is similar.
\end{proof}

\begin{corollary}
If $n$ is odd, any sharbly cocycle $\mu$ has even parity.
\end{corollary}

\begin{proof}
Since $n$ is odd, $-I_n$ generates $\GL_n(\Z)/ \SL_n(\Z)$.  But $-I_n$ acts trivially on $Sh_\bullet$ and hence trivially on 
$(Sh_\bullet)_{\SL_n(\Z)}$.  Therefore $s_n$
acts trivially on $(Sh_\bullet)_{\SL_n(\Z)}$.
\end{proof}

 \section{Next step: $n=4$}
 Continuing the notation of the preceding section, let 
 $z_4=[z_3|z_1]_{\SL_4(\Z)}$ and $\mu_4=\mu_3\times\mu_1$.  Since both $\mu_3$ and $\mu_1$ have even parity, $\mu_4$ makes sense.
 \begin{theorem}
(1) $ \mu_4(z_4)\ne0$.
(2) $H_3(\SL_4(\Z),St)\ne0$.
(3) $H^3(\SL_4(\Z),\Q)\ne0$.
 \end{theorem}
 
 \begin{proof}
 (1) Compute
 \[
 x=
 (\mu_3\times\mu_1)\left(
 \begin{bmatrix}
 1&0&0&1&0&1&0\\
0&1&0&-1&1&0&0\\
0&0&1&0&-1&-1&0\\
0&0&0&0&0&0&1
 \end{bmatrix}
 \right).
 \]
 The only pliable set of 6 vectors on the right hand side is that consisting of the first 6.  All other sets of 6 vectors span a space of dimension 4.  So
 \[
 x=
 \mu_3(z_3)\mu_1(z_1)\ne0.
 \]
 (2) Because $z_1$ represents a class in $H_0(\SL_1(\Z),St)$ and 
 $z_3$ represents a class in $H_3(\SL_3(\Z),St)$, it follows that
 $z_4$ represents a class in $H_{0+3}(\SL_4(\Z),St)$.
 Then (3) follows by Borel-Serre duality.
 \end{proof}
 
\begin{remark}
(1) This result is consonant with the known fact that
$H^i(\SL_4(\Z),\Q)$ equals 0 unless $i=0,3$, in which cases it is one-dimensional.  (2) With more trouble, one could show in a similar way that 
$(\mu_1\times\mu_3)[z_1|z_3]_{\SL_4(\Z)}\ne0$, but by the result just referred to, this would not give anything new in cohomology.  In fact, by Theorem~\ref{thm:comp} (3), 
$[z_1|z_3]_{\SL_4(\Z)}=-[z_3|z_1]_{\SL_4(\Z)}$.
\end{remark}

\section{Some cases that don't compose well: $n=2,5$}
Let $n=2$ and
\[
 z_2:=\begin{bmatrix}
 1&0&1\\
  0&1&-1\\
 \end{bmatrix}_{\SL_2(\Z)}.
  \]
  In~\cite{AGMpuzzle} it is proven that $z_2$ is a cycle (which is not hard to do by hand) and that there exists a sharbly cocycle $\mu_2$ (again essentially the volume) such that $\mu_2(z_2)\ne0$. 
  
  But there is a big difference between $n=2$ and $n=3$, namely that $2$ is even.  In fact, $z_2$ has odd parity, so that if it is composed with any even parity sharbly cycle, the result is 0.  For example,   $[z_2|z_3]_{\SL_5(\Z)}=0$, which had better be the case, because $H^6(\SL_5(\Z),\Q)=0$.
  
  To see that $z_2$ has even parity:
  \begin{gather*}
  s_2z_2=
  \begin{bmatrix}
 -1&0&-1\\
  0&1&-1\\
 \end{bmatrix}_{\SL_2(\Z)}=
 \begin{bmatrix}
 1&-1\\
 0&1\\
 \end{bmatrix}
 \begin{bmatrix}
 1&0&1\\
  0&1&1\\
 \end{bmatrix}_{\SL_2(\Z)}=\\
 \begin{bmatrix}
 1&-1&0\\
  0&1&1\\
 \end{bmatrix}_{\SL_2(\Z)}=
 -z_2.
  \end{gather*}
  
We also get 0 if we compose $z_2$ with itself.  This is because $z_2$ has 3 columns and 3 is an odd number.  The permutation that switches the first 3 columns and the last 3 columns of 
  $[z_2|z_2]_{\SL_4(\Z)}$ is odd, so it multiplies the value of the symbol by -1.  But multiplying  $[z_2|z_2]_{\SL_4(\Z)}$ on the right by the unimodular matrix
\[
\begin{bmatrix}
0&I_2\\
 I_2&0\\
 \end{bmatrix}
\]
also switches the first 3 columns and the last 3 columns and doesn't change the value of the sharbly.  So the sharbly must be equal 0.

It is known that $H^i(\SL_5(\Z),\Q)$ is 0 unless $i=0,5$ in which cases it is one-dimensional.  Therefore it should happen 
that $[z_4|z_1]_{\SL_4(\Z)}=0$, for otherwise we would obtain nonvanishing of $H_{3}(\SL_5(\Z),St))\approx H^7(\SL_5(\Z),\Q)$.
Indeed, we can see easily that $[z_4|z_1]_{\SL_4(\Z)}=0$.  Using Theorem~\ref{thm:comp} (3) we see that
\[
[z_4|z_1]_{\SL_4(\Z)}=[z_3|z_1|z_1]_{\SL_4(\Z)}=-[z_3|z_1|z_1]_{\SL_4(\Z)}
=-[z_4|z_1]_{\SL_4(\Z)}.
\]
More generally, if we compose any sharbly cycle with $z_1$ and compose the result with $z_1$, we obtain 0.

 \section{Next cases: $n=6,7$}
Since $z_4$ has 7 columns, which is an odd number, if we compose $z_4$ with itself we get 0.  So these methods will not obtain any nonzero classes for $\SL_8(\Z)$.
 We can however compose $z_3$ with itself or with $z_4$.
 
 To go further it is helpful to make the following definitions.  Angle brackets denote $\Q$-span and $S_k(A)$ denotes the set of all subsets of the set $A$ with $k$-elements.
 \begin{definition}
 Let $A$ be a set of nonzero vectors in $\Q^n$ of size $p$.  The \emph{depth-chart} of $A$ is the function 
$
 d_A:\Z\to\Z
$
 defined by 
 \[
 d_A(k)= \min_{S\in S_k(A)} \dim <S>
 \]
  if
 $0\le k\le p$ and $=0$ otherwise.
 \end{definition}
 Note that $d_A(k)$ is a weakly monotonic increasing function of $k$.
 
  \begin{definition}
 Let $A$ and $B$ be nonempty subsets of nonzero vectors in $\Q^n$ with $p$ and $q$ elements respectively. Then if $0\le i\le p$,
 \[
 A_i(B)
 \]
 denotes the set of sets obtained from $A$ by replacing exactly $i$ of its elements by $i$ pairwise distinct elements of $B$.  If $i<0$ or $i>p$, then it denotes the empty set.
   \end{definition}
 
 The proof of the following lemma is clear:
 \begin{lemma}\label{depth}
 Let $\Q^n=X\oplus Y$,
and suppose $A\subset X-\{0\}$ has $p$ elements, $B\subset Y-\{0\}$ has $q$ elements.  Then
\[
d_A(p-i)+d_B(i)= \min_{S\in A_i(B)} \dim S
\]
and
\[
d_{A\cup B}(k)=\min_{0\le \alpha \le k} (d_A(k-\alpha)+d_B(\alpha)).
\]
  \end{lemma}

   \begin{lemma}\label{key}
Suppose $z_1=[v_1,\dots,v_k]_{\SL_a(\Z)}$ 
  is a basic sharbly for $\SL_a(\Z)$, and  $z_2=[w_1,\dots,w_m]_{\SL_b(\Z)}$ basic sharbly for $\SL_b(\Z)$.  
  Let $z=[z_1|z_2]_{SL_n(\Z)}$, where $n=a+b$.  Let $A$ be the set of vectors to the left of the vertical bar, and $B$ the set to the right.

If for $i=1,\dots,k-1$
\[
 \min_{S\in A_i(B)} \dim S > a
\]
then the only possible pliable sets of vectors in $A\cup B$ are $A$ and subsets of $B$.
    \end{lemma}
  
   \begin{proof}
A pliable set of vectors for $z$ is a subset $C$ of  $A\cup B$ of dimension $a$.  View $C$ as made of $i$ vectors from $A$ and $k-i$ vectors from $B$.  Then the set of candidates for $C$ is $A_i(B)$.  If $i\ne0,k$, then
none of these $C$'s have dimension $a$, and the lemma follows.
  \end{proof}

 The following theorems in this and the next section are compatible with the computations mentioned in the introduction.
  \begin{theorem}\label{n6}
(1) $(\mu_3\times\mu_3)([z_3|z_3]_{\SL_6(\Z)})\ne0$.
(2) $H_6(\SL_6(\Z),St)\ne0$.
(3) $H^9(\SL_6(\Z),\Q)\ne0$.
 \end{theorem}
 
 \begin{proof}
 As before, (2) and (3) follow from (1).  For
 (1), let $E$ 
 be the columns of 
 \[
\begin{bmatrix}
 1&0&0&1&0&1\\    
0&1&0&-1&1&0 \\  
0&0&1&0&-1&-1
\end{bmatrix}. 
\]
Then the depth-chart of $E$ is 
\[
\begin{tabular}{l || lllllll }
 $i$& 0 &1  &2  &3&4  &5  &6  \\   \hline
 $d_E(i)$ & 0 &1  &2  &2 &3  &3  &3 \\
\end{tabular}
\]

Let $A$ be the columns of 
 \[
\begin{bmatrix}
 1&0&0&1&0&1\\    
0&1&0&-1&1&0 \\  
0&0&1&0&-1&-1\\ 
 0&0&0&0&0&0\\   
  0&0&0&0&0&0\\ 
   0&0&0&0&0&0\\ 
 \end{bmatrix}
 \] and 
 let $B$ be the columns of 
 \[
\begin{bmatrix}
 0&0&0&0&0&0\\   
  0&0&0&0&0&0\\ 
   0&0&0&0&0&0\\
 1&0&0&1&0&1\\    
0&1&0&-1&1&0 \\  
0&0&1&0&-1&-1\\ 
 \end{bmatrix}
 \] 
 The depth-charts of $A$ and $B$ are the same as the depth-chart of $E$.  By Lemma~\ref{depth}, 
 \[
\delta_i:= \min_{S\in A_i(B)} \dim S = d_A(6-i)+d_B(i)=d_E(6-i)+d_E(i).
\]
So
\[
\begin{tabular}{l || lllllll }
 $i$& 0 &1  &2  &3&4  &5  &6  \\   \hline
 $d_E(6-i)$ & 3 &3  &3  &2 &2  &1  &0 \\
  $d_E(i)$ & 0 &1  &2  &2 &3  &3  &3 \\
$\delta_i$ & 3 &4  &5  &4 &5  &4  &3 \\
\end{tabular}
\]
Apply Lemma~\ref{key} with $k=6$ and $a=3$, we see that the only pliable subsets of vectors for $[z_3|z_3]_{\SL_6(\Z)}$ are $A$ and $B$.
Now $[z_3|z_3]_{\SL_6(\Z)}=[A,B]_{\SL_6(\Z)}$. 
Therefore, in the formula for 
$(\mu_3\times\mu_3)([z_3|z_3]_{\SL_6(\Z)})$ there are two terms. The first term, corresponding to the trivial permutation,
is $\mu_3(z_3)\mu_3(z_3)$, which is a nonzero number, call it $t$.

 The other term corresponds to the shuffle permutation 
that switches $A$ and $B$, which is even, because they each have 6 vectors and 6 is even.  Let 
\[
\gamma=\begin{bmatrix}
0&-I_3\\
I_3&0\\
\end{bmatrix}.
\]
Then the term in question equals
\[
\mu_3([\gamma B]_{\SL_3(\Z)})\mu_3([\pi(\gamma A)]_{\SL_3(\Z)}) =
\mu_3([A]_{\SL_3(\Z)})\mu_3([\pi B]_{\SL_3(\Z)})
=\mu_3(z_3)\mu_3(z_3).
\]
The result is again $t$.  So 
$(\mu_3\times\mu_3)([z_3|z_3]_{\SL_6(\Z)})=2t\ne0$.
\end{proof}

 \begin{theorem}\label{n7}
(1) $(\mu_3\times\mu_4)[z_3|z_4]_{\SL_6(\Z)})\ne0$.
(2) $H_6(\SL_7(\Z),St)\ne0$.
(3) $H^{15}(\SL_7(\Z),\Q)\ne0$.
 \end{theorem}
 
 \begin{proof}
 As before, (2) and (3) follow from (1).  For
 (1),  let $F$ 
 be the columns of 
 \[
\begin{bmatrix}
 1&0&0&1&0&1&0\\    
0&1&0&-1&1&0&0 \\  
0&0&1&0&-1&-1&0\\
0&0&0&0&0&0&1\\
\end{bmatrix}. 
\]
Then the depth-chart of $F$ is 
\[
\begin{tabular}{l || llllllll }
 $i$& 0 &1  &2  &3&4  &5  &6 &7 \\   \hline
 $d_F(i)$ & 0 &1  &2  &2 &3  &3  &3&4 \\
\end{tabular}
\]

Let $A$ be the columns of 
 \[
\begin{bmatrix}
 1&0&0&1&0&1\\    
0&1&0&-1&1&0 \\  
0&0&1&0&-1&-1\\ 
 0&0&0&0&0&0\\   
  0&0&0&0&0&0\\ 
   0&0&0&0&0&0\\ 
 \end{bmatrix}
 \] and 
 let $B$ be the columns of 
 \[
\begin{bmatrix}
 0&0&0&0&0&0&0\\   
  0&0&0&0&0&0&0\\ 
   0&0&0&0&0&0&0\\
 1&0&0&1&0&1&0\\    
0&1&0&-1&1&0 &0\\  
0&0&1&0&-1&-1&0\\ 
 0&0&0&0&0&0&1\\ 
 \end{bmatrix}
 \] 
 The $A$ and $E$ have the same depth-charts; $B$ and $F$ have the same depth-charts.  By Lemma~\ref{depth}, 
 \[
\delta_i:= \min_{S\in A_i(B)} \dim S = d_A(6-i)+d_B(i)=d_E(6-i)+d_F(i).
\]
So
\[
\begin{tabular}{l || lllllll }
 $i$& 0 &1  &2  &3&4  &5  &6  \\   \hline
 $d_E(6-i)$ & 3 &3  &3  &2 &2  &1  &0 \\
  $d_F(i)$ & 0 &1  &2  &2 &3  &3  &3 \\
$\delta_i$ & 3 &4  &5  &4 &5  &4  &3 \\
\end{tabular}
\]
Apply Lemma~\ref{key} with $k=6$ and $a=3$, we see that the only pliable subsets of vectors for $[z_3|z_4]_{\SL_7(\Z)}$ are $A$ and some set(s) of 6 vectors of $B$.  In fact, the only set that works are the first 6 vectors of $B$. 
Now $[z_3|z_4]_{\SL_7(\Z)}=[A,B]_{\SL_7(\Z)}$. 
Therefore, in the formula for 
$(\mu_3\times\mu_4)([z_3|z_4]_{\SL_7(\Z)})$ there are two terms. The first term, corresponding to the trivial permutation,
is $\mu_3(z_3)\mu_4(z_4)$, which is a nonzero number, call it $t$.

 The other term corresponds to the shuffle permutation 
that switches the 6 vectors of $A$ and the first 6 vectors of $B$, which is an even
permutation.  
Let $B'$ denote the first 6 vectors of $B$, and let $e_7$ be the last vector of $B$.  Let $A'$ be the 6 vectors of $A$ plus $e_7$.
Let 
\[
\gamma=\begin{bmatrix}
0&-I_3&0\\
I_3&0&0\\
0&0&1
\end{bmatrix}.
\]
Then the term in question equals 
\[
\mu_3([\gamma B']_{\SL_3(\Z)})\mu_4([\pi(\gamma A')]_{\SL_4(\Z)}) =
\mu_3([A]_{\SL_3(\Z)})\mu_3([\pi B]_{\SL_4(\Z)})
=\mu_3(z_3)\mu_4(z_4).
\]
The result is again $t$.  So 
$(\mu_3\times\mu_4)([z_3|z_4]_{\SL_7(\Z)})=2t\ne0$.
 
 \end{proof}

 \section{Larger $n$}
  We can fruitfully compose $z_3$ with itself any number of times and 
 we can also compose the result with $z_1$, or equivalently use one fewer $z_3$ and compose the result with $z_4$.
 
  \begin{definition}
 Let $k\ge1$.  Define $z_{3k}$ by induction: $z_3$ and $z_4$ have already been defined and if $k>1$,
 set 
 \begin{gather*}
  z_{3k}=[z_3|z_{3k-3}]_{\SL_{3k}(\Z)},\\
z_{3k+1}=[z_3|z_{3k-2}]_{\SL_{3k+1}(\Z)}\\
  \mu_{3k}=\mu_3\times \mu_{3k-3}\\
  \mu_{3k+1}=\mu_3\times \mu_{3k-2}.
\end{gather*}
 \end{definition}
 
 \begin{lemma}\label{1}
 Let $k\ge 1$ and let $C$ be either the set of vectors that appear in $z_{3k}$ or that appear in $z_{3k+1}$.
Then the depth chart of $C$ begins like
\[
\begin{tabular}{l || llllllll }
 $i$& 0 &1  &2  &3&4  &5  &6&\dots  \\   \hline
 $d_C(i)$ & 0 &1  &2  &2 &3  &3  &3&\dots \\
\end{tabular}
\]
 \end{lemma}
 \begin{proof}
 The statement is true for $k=1$.  Now assume it is true for all values less than $k$.
  Let $A$ be the set of vectors that appear in $z_3$ and
 $B$ be the set of vectors that appear in $z_{3k-3}$ (resp. $z_{3k-2}$),
 so that $C=A\cup B$.
 By Lemma~\ref{depth}, 
\[
d_{C}(i)=\min_{0\le \alpha \le i} d_A(i-\alpha)+d_B(\alpha).
\]
We only need to consider $0\le i,\alpha\le 6$.  The claimed result easily follows.
\end{proof}

 \begin{theorem}\label{n3k} Let $k\ge 1$.\ 
 
(1) $(\mu_3\times\mu_{3k})[z_3|z_{3k}]_{\SL_{3k+3}(\Z)})\ne0$.

(2) $H_{3k+3}(\SL_{3k+3}(\Z),St)\ne0$.

(3) $H^{(9/2)(k^2+k)}(\SL_{3k+3}(\Z),\Q)\ne0$.

(4) $(\mu_3\times\mu_{3k+1})[z_3|z_{3k+1}]_{\SL_{3k+4}(\Z)})\ne0$.

(5) $H_{3k+3}(\SL_{3k+4}(\Z),St)\ne0$.

(6) $H^{(9k^2+15k)/2+3}(\SL_{3k+4}(\Z),\Q)\ne0$.
 \end{theorem}
 
 \begin{proof}
 As before, we only need to prove (1) and (4).    Call (1) the ``first case'' and (4) the ``second case''.
 Let $z$ stand for either $z_{3k}$ or $z_{3k+1}$, $\mu$ stand for either $\mu_{3k}$ 
 or $\mu_{3k+1}$, and $n$ for $3k+3$ or $3k+4$,
  in the first and second cases, respectively.  
  
  Let $e_1,\dots,e_n$ be the standard basis of $\Q^n$.  Let
  \[
  V_0=<e_1,e_2,e_3>, V_1= <e_4,e_5,e_{6}>,\dots,
  V_k=<e_{3k+1},e_{3k+2},e_{3k+3}>.
  \]
Then $\Q^n=\oplus_{\beta=0}^k V_\beta$ in the first case and
   $\Q^n=\oplus_{\beta=0}^k V_\beta\oplus \Q e_n$ in the second case.
   Let $\iota_\beta:\Q^3\to \Q^n$ be the embedding that sends $e_1$ to $e_{3\beta+1}$, $e_2$ to $e_{3\beta+2}$,  and $e_3$ to $e_{3\beta+3}$.
  The columns of $z$ come in sets of 6, each of which is the image under 
  $\iota_\beta$ of the set of columns of $z_3$, for 
  $\beta=0,\dots,k$, plus one more column, namely $e_n$, in the second case.  Let $S_\beta$ be the set of 6 columns in $z$ which lie in $V_\beta$.  
  For $\beta=0, \dots,k$, 
  let $\sigma_\beta$ denote the shuffle permutation that swaps $S_0$ and $S_\beta$.  (So that $\sigma_0$ is the identity permutation.)
  
  When computing 
  $(\mu_3\times\mu)[z_3|z]_{\SL_n(\Z)})$,
  according to the definition, by Lemma~\ref{1}, we only need to consider permutations that exchange all 6 columns of $z_3$ with some choice of 6 columns of $z$.
  In fact, one sees easily that 
  the only pliable sets of columns of $[z_3|z]_{\SL_n(\Z)}$
  are the $S_\beta$.  
  Let $\mu_3([z_3]_{\SL_3(\Z)})\times\mu([z]_{\SL_{n-3}(\Z)})=t$.  
By induction $t\ne0$.
  Just as in the proofs of 
  Theorems~\ref{n6} and~\ref{n7}, we see that the term in the sum for
     $(\mu_3\times\mu)[z_3|z]_{\SL_n(\Z)})$
     involving $\sigma_\beta$ is also equal to $t$.  Therefore
  $(\mu_3\times\mu)([z_3|z]_{\SL_{n}(\Z)})=(k+1)t\ne0$.
 \end{proof}

\bibliographystyle{amsalpha}
\bibliography{Ash-cosharblies}

\end{document}